\documentclass[a4paper,12pt]{amsart}

\usepackage{amsmath}
\usepackage{amssymb}
\usepackage{amsthm}
\usepackage{amscd}
\usepackage{enumerate}
\usepackage{xy}
\usepackage{xypic}
\usepackage{geometry}

\newtheorem{theorem}{Theorem}[section]

\newtheorem{proposition}[theorem]{Proposition}
\newtheorem{lemma}[theorem]{Lemma}
\newtheorem{corollary}[theorem]{Corollary}
\newtheorem{definition}[theorem]{Definition}
\newtheorem{remark}[theorem]{Remark}
\newtheorem{example}[theorem]{Example}

\newcommand{\holo}{\mathcal{O}}

\newcommand{\C}{\mathbb{C}}

\newcommand{\Z}{\mathcal{Z}}
\newcommand{\J}{\mathcal{J}}

\newcommand{\F}{\mathcal{F}}

\newcommand{\dbar}{\bar{\partial}}
\newcommand{\Res}{\operatorname{Res}}

\newcommand{\E}{\mathcal{E}}
\newcommand{\End}{\operatorname{End}}
\newcommand{\Hom}{\operatorname{Hom}}

\title{A local duality principle for ideals of pure dimension}
\author{Johannes Lundqvist}
\date{}

\begin{document}

\maketitle

\begin{abstract}
We prove that a certain cohomological residue associated to an ideal of pure
dimension is annihilated exactly by the ideal. The cohomological residue is
quite explicit and generalizes the classical local Grothendieck residue and the
cohomological residue of Passare.
\end{abstract}

\section{Introduction}\label{intro}

Assume that $f^1,\ldots,f^n$ are germs of holomorphic functions at
$x\in\mathbb{C}^n$ and that they define a complete intersection, i.e., that their common
zero set consists of the single point $x$.
Let $B(z)$ be the Bochner-Martinelli kernel at $x$, i.e.,
\begin{equation*}
B(z)=\frac{\sum(-1)^{j-1}(\bar{z}_j-\bar{x}_j) d\bar{z}_1\wedge\ldots\wedge
\widehat{d\bar{z}_j}\wedge\ldots\wedge d\bar{z}_n}{(|z_1-x_1|^2+\ldots+|z_n-x_n|^2)^n},
\end{equation*}
and let $f$ be the mapping from $\C^n_x$ to $\C^n_x$ given by $(f^1, \ldots, f^n)$.
Then $u:=f^*B$ is defined and $\dbar$-closed in a punctured neighborhood of $x$
and it turns out that $\varphi u$ is $\dbar$-exact if and only if
$\varphi\in\J_x=\langle f^1,\ldots, f^n\rangle$, the ideal in the local ring $\holo_x$ generated by the germs $f_j$.
Another way of putting this is that the Grothendieck residue
\begin{equation*}
\Res_f(\psi):=\left( \frac{1}{2\pi i}\right)^n\int_{\partial D} u\wedge\psi,
\quad \psi\in\Omega_x^n,
\end{equation*}
where $D$ is a sufficiently small neighborhood (depending on $\psi$) of $x$, is annihilated by
$\varphi$, i.e., $\varphi\Res_f(\psi):=\Res_f(\varphi\psi)=0$ for all $\psi$, if
and only if $\varphi$ belongs to $\J_x$. This is called the local duality principle,
see \cite[page 655 and 659]{GH}.
Note that the integral only depends on the Dolbeault cohomology class, in a punctured neighborhood of $x$, defined by
$u$. This cohomology class can be defined in a more abstract way but the interesting point
here is that we have a concrete condition on $\varphi$ for membership in $\J_x$.

The local duality principle above extends to more general ideals.
Assume that $\J_x=\langle f^1,\ldots,f^p\rangle$ is a complete intersection and
let $\Z$ be the common zero set of $\J_x$, i.e., 
\[
\Z=\{z\in\mathcal{U};\; f^1(z)=\ldots=f^p(z)=0\},
\]
where $\mathcal{U}$ is a small neighborhood of $x$.
Then Passare found , see \cite[Theorem 6.3.1]{Passare}, that in
$\mathcal{U}\setminus\Z$ the form $f^*B$, where $B$ is the Bochner-Martinelli form on
$\C^p$, works just like when $p=n$.
That is, $\varphi\in\J_x$ if and only if
\begin{equation}\label{IntCond}
\int\varphi u \wedge\dbar\psi =0,
\end{equation}
for all smooth $(n,n-p)$-forms $\psi$ with compact support in a sufficiently
small neighborhood of $x$ and such that $\dbar\psi=0$ close to $\Z$. Note that when $n=p$
the condition \eqref{IntCond}, even though it looks a bit different, coincides
with the condition for the local duality principle by
Stokes' theorem.
As before the integral only depends on the Dolbeault cohomology class of $u$ and
$[\varphi u]_{\dbar}=0$ outside of $\Z$ if and only if $\varphi\in\J_x$, see \cite[Corollary
6.3.4]{Passare}.

It is well known, see \cite{Roos}, that if the ideal $\J_x\subset\holo_x$ has pure codimension $p$,
then $\varphi\in\holo_x$ annihilates the $\holo_x$-module
$\E xt^p_{\holo_x}(\holo_x/\J,\holo_x)$ if and only if $\varphi\in\J_x$.
Let $\mathcal{H}om_{\holo_x}(\holo_x/\J_x, \mathcal{CH}_{\Z})$ be the $\holo_x$-module of Coleff-Hererra currents in the sense of Bj\"ork, \cite{Bjork}, annihilated by $\J_x$, and let $\mathcal{H}om_{\holo_x}(\holo_x/\J_x, \mathcal{H}_{\dbar}^{0,p-1}(\C^n_x\setminus\Z))$ be the
$\holo_x$-module of Dolbeault cohomology classes of bidegree $(0,p)$, defined outside of
$\Z$, annihilated by $\J_x$.
There is a commutative diagram
\begin{equation}\label{tau}
\xymatrix{
\E xt_{\holo_x}^p(\holo_x/\J_x,\holo_x) \ar[r]^(.4){\tau} \ar[dr]_{\cong}^{\nu}
&\mathcal{H}om_{\holo_x}(\holo_x/\J_x, \mathcal{H}_{\dbar}^{0,p-1}(\C^n_x\setminus\Z)) \\
&\mathcal{H}om_{\holo_x}(\holo_x/\J_x, \mathcal{CH}_{\Z}), \ar[u]_{\gamma}
}
\end{equation}
where $\gamma$ is defined in the following way:
If $\mu\in\mathcal{H}om_{\holo_x}(\holo_x/\J_x, \mathcal{CH}_{\Z})$,
then we can assume that there exists a small neighborhood of $x$ where $\mu$ is defined. The image $\gamma(\mu)$ is then defined to be the class of $\eta$, where $\eta$ is any solution to $\dbar \eta =\mu$ outside $\Z$ and in a possible smaller neighborhood of $x$.
The mapping $\nu$ is a well known isomorphism and $\tau$ is defined to make \eqref{tau} commute.
It turns out that $\gamma$ is injective, see Theorem 6.1 in \cite{Andersson2}. This means that if $\xi_j$ are generators of $\E xt_{\holo_x}^p(\holo_x /\J, \holo_x)$, then the images $v_l:=\tau(\xi_j)$ satisfy a local duality principle. That is, $\varphi v_{\ell}=0$ for all $\ell$ if and only if
$\varphi\in\J_x$.

In \cite{Andersson}, there is an explicit realization of the isomorphism $\nu$, based on the work in \cite{AnderssonWulcan}. This realization can be used to get explicit generators $v_{\ell}$ in $\mathcal{H}om_{\holo_x}(\holo_x/\J_x, \mathcal{H}_{\dbar}^{0,p-1}(\C^n_x\setminus\Z))$, given that $\xi_j$ are explicit.
In the special case where $\J_x$ is a complete intersection there is only one
generator for $\E xt_{\holo_x}^p(\holo_x/\J_x,\holo_x)$ and the image under $\nu$
is, up to a holomorphic factor, the classical Coleff-Herrera product, see
\cite{CH} for the definition,  and the image under $\tau$ is the
cohomological residue $f^*B$ of Passare. The fact that the Coleff-Herrera
product satisfies a local duality principle was proved independently by Passare
and Dickenstein-Sessa in \cite{Passare} and \cite{DickensteinSessa}
respectively.

The theory of Coleff-Hererra currents relies on Hironaka's desingularization theorem.
In this paper we take inspiration of the explicit realization of $\nu$ and give an algebraic 
construction of the cohomology classes $v_l$ and a rather elementary proof that they satisfies a duality principle. The construction and proof do not rely on the desingularization theorem of Hironaka or residue currents.
We also give an integral condition similar to \eqref{IntCond}.
The case where $\J_x$ is Cohen-Macaulay is already dealt with in \cite{L}.
In \cite{Lej1} there is a related cohomology class in the Cohen-Macaulay case.
However, that class arises in a more abstract way not giving an explicit
representative.
\\
\\
{ \bf Acknowledgement: } The author would like to thank Mats Andersson for
suggesting the problem and Elizabeth Wulcan for valuable comments that helped
improve this paper.


\section{A cohomological residue associated to ideals with pure dimension}

Let $\J_x$ be an ideal in the local ring $\holo_x$ of germs of holomorphic
functions at $x\in\mathbb{C}^n$ and assume that $\J_x$ has pure dimension, i.e.,
all associated primes have the same dimension.
Let $X$ be a small Stein neighborhood of $x$ and let $\J$ be the ideal sheaf over $X$ defined from $\J_x$.
Let $E_0$ be a trivial vector bundle of rank $1$ over $X$ and let 
\begin{equation}\label{resolution}
0\to\holo(E_N)\overset{f_N}{\to}\ldots\overset{f_3}{\to}\holo(E_2)\overset{f_2}{
\to}\holo(E_1)\overset{f_1}{\to}\holo(E_0)
\end{equation}
be a locally free resolution of the sheaf $\holo(E_0)/\J$, where $\J$ is
considered as a subsheaf of the sheaf of sections $\holo(E_0)=\holo$.

Set $E=\bigoplus E_k$.
We introduce a super structure on
$E$, i.e., we let $E=E^+\oplus E^-$ where $E^{+}=\bigoplus_k E_{2k}$ and
$E^{-}=\bigoplus_k E_{2k+1}$.
We say that elements in $E^{+}$ have even order and that elements in $E^{-}$
have odd order. The grading on $E$ induces a natural $\mathbb{Z}_2$-grading on
$\End E =\End E^{+}\oplus \End E^{-}$, where 
\begin{equation*}
\End E^{+}= \Hom(E^{+},E^{+})\oplus \Hom(E^{-},E^{-})
\end{equation*}
and
\begin{equation*}
\End E^{-}=\Hom(E^{+},E^{-})\oplus \Hom(E^{-},E^{+}). 
\end{equation*}
Let $\Z$ be the zero locus of $\J$ and let $\mathcal{M}$ denote the left $\E(\Lambda T^*(X\setminus\Z))$-module $\mathcal{E}( \Lambda T^*(X\setminus\Z) \otimes E)$.
We let the operator $f:=\sum f_j$ on $\E(E)$ and the operator $\dbar$ on $\mathcal{E}(\Lambda T^*(X\setminus\Z))$ act on $\mathcal{M}$ by the definitions
\[
f(\xi\eta):=(-1)^{\deg\xi\deg\eta}\xi f(\eta)
\]
and
\[
\dbar(\xi\eta):=(-1)^{\deg\xi\deg\eta}\dbar\xi \eta,
\]
for $\xi\in\E(\Lambda T^*(X\setminus\Z))$ and $\eta\in\E(E)$. It then follows that 
\begin{equation}\label{fdbar}
\dbar\circ f = -f\circ\dbar.
\end{equation}
Following the ideas in \cite{AnderssonWulcan} we introduce
the operator $\nabla_E:=f-\dbar$. The operator $\nabla_E$ acts on $\mathcal{M}$ but it also extends to a
map $\nabla_{\End E}$ on the endomorphisms of $\mathcal{M}$ by the
definition $\nabla_{\End E}v:=\nabla_E\circ v + v\circ \nabla_E$. This implies
that
\begin{equation}\label{NablaSol}
\nabla_E(v\varphi)=(\nabla_{\End E} v)\varphi - v\nabla_E(\varphi),
\end{equation}
for smooth $(p,q)$-sections $\varphi$ of  $E$. Note that $\eqref{fdbar}$ implies that $\nabla_E^2=0$, and hence also that
$\nabla_{\End E}^2=0$.

Given a Hermitian metric on $E$  let $\sigma$ be the minimal inverse to $f$ on $X\setminus\Z$,
i.e., $\sigma\varphi=\eta$, where $\eta$ is the solution to $f\eta=\varphi$ with
minimal norm, if $\varphi$ belongs to $\operatorname{Im} f$ and $\sigma\varphi=0$ if $\varphi$ belongs to
$(\operatorname{Im} f)^\perp$.
It is easy to check that $f\sigma+\sigma f = \operatorname{id}$ and that
$\sigma^2=0$ on $E$. 

The endomorphism 
\begin{equation}\label{u}
u:=\sigma+\sigma(\dbar\sigma)+\sigma(\dbar\sigma)^2+\ldots
\end{equation}
on $\mathcal{M}$ satisfies 
\begin{equation}\label{nablau}
\nabla_{\End E}u=\operatorname{id},
\end{equation} 
the identity on $\mathcal{M}$, see \cite{AnderssonWulcan}.
Assume from now on that $\varphi\in\holo(E_0)$. Then $\nabla_E \varphi = 0$ and as
a direct consequence of \eqref{NablaSol} and \eqref{nablau} we get
\begin{equation}\label{NablaEq}
\nabla_E(u\varphi)=\varphi.
\end{equation}

To define the cohomological residue we take inspiration from \cite{Andersson} and consider the dual complex of \eqref{resolution},
\begin{equation*}
0\longrightarrow\holo(E_0^*)\overset{f_1^*}{\longrightarrow}\holo(E_1^*)\overset
{f_2^*}{\longrightarrow}\ldots\overset{f_{p-1}^*}{\longrightarrow}
\holo(E^*_{p-1})\overset{f_p^*}{\longrightarrow}\holo(E^*_p)\overset{f_{p+1}^*}{
\longrightarrow}\ldots.
\end{equation*}
Let $u_p$ be the component of $u$ that take values in $\Hom(E,E_p)$.
If $\xi\in\holo(E_p^*)$ is such that
$f_{p+1}^*\xi=0$, then 
\begin{equation*}
\dbar(\xi u_p\varphi)=\xi\dbar( u_p\varphi )= \xi f_{p+1} (u_{p+1}\varphi) =
(f_{p+1}^* \xi) u_{p+1}\varphi = 0,
\end{equation*}
where we have used \eqref{NablaEq} in the second equality.
Also, if $\xi-\xi^{\prime}=f^*_p\eta$ we see that
\[
\xi u_p\varphi - \xi^{\prime} u_p\varphi = (f^*_p \eta) u_p\varphi = \eta f_p (
u_p\varphi)=\dbar(\eta u_{p-1}\varphi).
\]
Let  $\mathcal{H}^p(E^*_{\bullet})$ be the $p$-th cohomology group of the
complex $\holo(E_{\bullet}^*)$, i.e., locally a representation of $\E xt_{\holo_x}^p(\holo_x/\J_x,\holo_x)$, and let  $\F$ denote the sheaf where
$\F_y=H_{\dbar}^{0,p-1}(\C^n_y\setminus\Z)$ for $y\in X$.
To conclude, we then get a bilinear map
\begin{equation}\label{map}
\mathcal{H}^p(E^*_{\bullet})\times \holo(E_0) \longrightarrow \F,
\end{equation}
defined by $([\xi],\varphi)\mapsto[\xi u_p\varphi]_{\dbar}$.

\begin{lemma}\label{BevisLemma}
Let  $\varphi\in\holo(E_0)$ and assume that  $\xi\in\holo(E_p^*)$ and 
$\phi,\phi^{\prime}\in\mathcal{M}$ satisfies  $f_{p+1}^*\xi=0$ and $\nabla_E \phi=\nabla_E \phi^{\prime} = \varphi$
respectively. Then $[\xi \phi_p]_{\dbar} = [\xi \phi^{\prime}_p]_{\dbar}$, where $\phi_p$ is the components of $\phi$ taking values in $E_p$. In particular $[\xi \phi_p]_{\dbar}$=$[\xi u_p \varphi]_{\dbar}$, where $u$ is defined by $\eqref{u}$.
\end{lemma}

\begin{proof}
Assume that $\nabla_E \phi = \nabla_E \phi^{\prime} =\varphi$.
Then $\nabla_E(\phi - \phi^{\prime})=0$ and hence 
\[
\nabla_E(u(\phi-\phi^{\prime}))= (\nabla_{\End E}u)(\phi-\phi^{\prime})-u\nabla_E(\phi-\phi^{\prime})=(\phi-\phi^{\prime})
\]
by \eqref{NablaSol}.
If we denote $u(\phi-\phi^{\prime})$ by $\psi$ we get that 
\[
f_{p+1}\psi_{p+1}-\dbar \psi_p = \phi_p - \phi^{\prime}_p,
\]
which implies that 
\[\xi \phi_p-\xi \phi^{\prime}_p =-\dbar\xi \psi_p.\]
\end{proof}

The following proposition says that the cohomology classes we get from $u$ essentially only depend on the ideal sheaf $\J$.

\begin{proposition}\label{well-defined}
Assume that $u$ is the endomorphism \eqref{u} associated to a resolution $\holo(E_{\bullet})$ of $\J$.
Then the following holds:
\begin{enumerate}[(i)]
 \item Let $\xi\in\holo(E_p^*)$ and assume that $f_{p+1}^*\xi=0$.
The cohomology class $[\xi u_p \varphi]_{\dbar}$ does not depend on the metric on $E$.
 \item\label{t} Assume that $\holo(F_{\bullet})$ is another resolution of the sheaf $\holo(E_0)/\J$ and let $v$ be the endomorphism associated with $\holo(F_{\bullet})$. Then, for appropriately chosen metrics,  there exists a
holomorphic isomorphism $g_0$, from $E_0$ to $F_0$, and a holomorphic isomorphism
$g_p$, from a subbundle of $E_p$ to a subbundle of $F_p$, such that
\begin{equation*}
v_p \varphi= g^{-1}_p w_p g_0 \varphi,
\end{equation*}
for $\varphi\in\holo(E_0)$.
\end{enumerate}
\end{proposition}

\begin{proof}
If $u$ and $u^{\prime}$ satisfy
$\nabla_E(u)=\nabla_E(u^{\prime})=\varphi\in\holo(E_0)$ but are defined from
different metrics on $E$, then they still satisfy $\nabla_E(u-u^{\prime})=0$. Hence
they define the same cohomology class by Lemma \ref{BevisLemma}. That is, the cohomology class is independent of the metric.

Assume now that $\holo(E_{\bullet})$ and $\holo(F_{\bullet})$ are two different
resolutions of $\holo(E_0)/\J$. Each of these resolutions is isomorphic to a direct sum
of a minimal resolution and a trivial complex, and minimal resolutions are
isomorphic, see \cite[Theorem 20.2]{Eisenbud}. We therefore have isomorphisms
$h^1,h^2$ and $g$ according to the diagram 
\begin{equation*}
\xymatrix{
E_{\bullet} \ar[r]^(.4){h^1} & E_{\bullet}^{\prime}\ar[d]<-3ex>^g\oplus
E_{\bullet}^{\prime\prime} \\
F_{\bullet} \ar[r]^(.4){h^2} & F_{\bullet}^{\prime}\oplus
F_{\bullet}^{\prime\prime}.
}
\end{equation*}
By assumption $E_0$ and $F_0$ have rank $1$, so we see that $E_0^{\prime\prime}$ and
$F_0^{\prime\prime}$ are equal to $0$. 
If $u^{\oplus}$ is the endomorphism associated to $E_{\bullet}^{\prime}\oplus
E_{\bullet}^{\prime\prime}$ and we choose a metric on $E_{\bullet}$ that
respects the direct sum, we can write $u^{\oplus}=u^{\prime}\oplus
u^{\prime\prime}$,  where $u^{\prime}$ and $u^{\prime\prime}$ are the endomorphisms
associated to $E_{\bullet}^{\prime}$ and $E_{\bullet}^{\prime\prime}$
respectively. Note that
$u^{\oplus}\varphi = u^{\prime} \varphi$. 
Assume that $u$ and $v$ satisfy the assumption in \eqref{t}. Then
 $u \varphi = (h^1)^{-1} u^{\prime}h^1 \varphi$. Analogously we get $v \varphi
= (h^2)^{-1} v^{\prime}h^2 \varphi$ and therefore
\begin{align*}
u_p \varphi = &(h^1)^{-1} u_p^{\prime}h^1 \varphi = 
(h^1)^{-1} g^{-1} v_p^{\prime} g h^1 \varphi = 
(h^1)^{-1} g^{-1} h^2 v_p^{\prime}(h^2)^{-1} g h^1 \varphi = \\
&(h^1)^{-1} g^{-1} h^2 v_p ((h^1)^{-1} g^{-1} h^2)^{-1}\varphi.
\end{align*}
Let $g_0$ be the part of $((h^1)^{-1} g^{-1} h^2)^{-1}$ that acts on $E_0$ and $g_p$ the part that acts on $E_p$. Then  $u_p \varphi= g^{-1}_p v_p g_0 \varphi$ and we are done.
\end{proof}

Let $\{[\xi_{\ell}]\}_{\ell=1}^r$ be a set of generators for 
$\mathcal{H}^p(E^*_{\bullet})$.
We denote the cohomology class 
$[\xi_{\ell} u_p 1]_{\dbar}$ by $\omega_{\ell}$.
The cohomology classes $\omega_{\ell}$ are the cohomological residues we are interested in, i.e., the ones annihilated by $\J_x$, see Theorem \ref{Thm} below.

Note that if $\varphi\in\holo(E_0)$, then the simple calculation
\[
\nabla_E(\varphi u1)=(\nabla_E\varphi )u1 + \varphi \nabla_E (u1)=\varphi\nabla_E (u1) =
\varphi(\nabla_{\End E}u1-u\nabla_E1) = \varphi
\]
and Lemma \ref{BevisLemma} implies that
\begin{equation}\label{Tek}
\varphi\omega_{\ell}=[\xi_{\ell}u_p\varphi].
\end{equation}

\section{A local duality principle}

As before, we let $x$ be a point in $\C^n$ and consider a small Stein neighborhood, $X$, of $x$.
Let $\J_x\subset\holo_x$ be an ideal and let $\J$ be the ideal sheaf over $X$ defined by $\J_x$.

Let $D\subset X$ be a small neighborhood of $x$ and denote by $\mathcal{D}_{\Z}$ the ring of smooth
$(n,n-p)$-forms $\psi$ with compact support in $D$ that satisfy $\dbar\psi=0$
close to $\Z$.

\begin{definition}
The residue 
\begin{equation*}
\Res^{\ell}_{\J}:\mathcal{D}_{\Z} \to\mathbb{C}
\end{equation*}
is given by
\begin{equation}
\Res^{\ell}_{\J}(\psi)=\int \omega_{\ell}\wedge\dbar\psi.
\end{equation}
\end{definition}

We see that the residue $\Res_{\J}^{\ell}$ is well-defined, i.e., it does not depend on
the choice of representative of the cohomology class, by Stokes' theorem.

We define multiplication of the residue with  $\varphi\in\holo(E_0)$ as the mapping
\begin{equation*}
\varphi\Res^{\ell}_{\J}(\psi):=
\Res^{\ell}_{\J}(\varphi\psi)=
\int \omega_{\ell}\wedge\dbar(\varphi\psi)=
\int \varphi\omega_{\ell}\wedge\dbar\psi.
\end{equation*}
We say that $\Res_{\J}^{\ell}=0$ at $x$ if there exists a small neighborhood $D$ of $x$ such that $\Res_{\J}^{\ell}=0$ in $\mathcal{D}_\Z$.

\begin{example}\label{ex}\rm{
If $\J_x$ is a complete intersection, then the Koszul complex is a
resolution of $\holo(E_0)/\J$. This means that $r=1$ and that $\omega_1$ coincides with the
cohomological residue defined by Passare in \cite{Passare}. In the special case
of a complete intersection of codimension $n$ the residue $\Res^{1}_{\J}$ 
coincides with the local
Grothendieck residue, cf., Section \ref{intro}. The proof of this fact can be
seen in Example 2.6 in \cite{L}.
}
\end{example}

In view of Example \ref{ex} the following theorem is a direct generalization of
the local duality principle for complete intersection ideals and the local duality
principle for the cohomological residue of Passare.

\begin{theorem}\label{Thm}
Let $\varphi\in\holo(E_0)$ and let $[\xi_{\ell}]$, $\ell=1,\ldots,r$, generate $\mathcal{H}^p(E^*_{\bullet})$. Then the following are equivalent:
\smallskip\begin{enumerate}[(i)]
\item\label{ett} \smallskip$\varphi\in\J_x$
\item\label{tva} \smallskip$\varphi\omega_{\ell} = 0$ for all $\ell=1,\ldots,r$
\item\label{tre} \smallskip$\varphi\Res_{\J}^{\ell}=0$ at $x$ for all $\ell=1,\ldots,r$.
\end{enumerate}
\end{theorem}

In light of \eqref{u} we get that 
$u_p=\sigma(\dbar\sigma)^{p-1}$, thus precisely as in the complete intersection case we
get an explicit integral condition that $\varphi$ must fulfill in order to be in
$\J_x$.
\begin{corollary}
Let $\varphi\in\holo(E_0)$, $D$ be a sufficiently small neighborhood of $x$, and let $[\xi_{\ell}]$, $\ell=1,\ldots,r$, generate $\mathcal{H}^p(E^*_{\bullet})$.
Then $\varphi\in\J_x$ if and only if
\begin{equation*}
\int \xi_{\ell} \sigma(\dbar\sigma)^{p-1}\varphi\wedge \dbar\psi = 0
\end{equation*}
for all $\ell=1,\ldots,r$, and all $\psi\in\mathcal{D}_{\Z}$.
\end{corollary}

\begin{remark}\label{gammal}
{\rm
If the ideal $\J_x$ is Cohen-Macaulay, then Theorem~\ref{Thm} coincides with Theorem~2.4 in \cite{L}.
In that case we may choose a free resolution of $\holo(E_0)\J$ that stops at level $p$, i.e., $E_{p+1}=0$. In particular, every $\xi\in\holo(E_p^*)$ satisfies $f_{p+1}^*\xi=0$. This means that we may chose $\xi_l$ as the vector with all entries equal to zero except entry $\ell$ that is equal to one, and that $u_p$ is a vector-valued $\dbar$-closed form. The cohomology classes $\varphi\omega_l$ are then represented by the varius entries in the vector $u_p\varphi$.
We use the Cohen-Macaulay case of the theorem in the proof of Theorem \ref{Thm}.
}
\end{remark}

\begin{lemma}\label{lemma}
Assume that 
$x\in\C^n$ and 
that $\mathcal{I}_x$ and $\J_x$ are ideal in $\holo_x$ such that
$\J_x\subset\mathcal{I}_x$. Moreover, assume that $\J_x$ has pure dimension and that there exists a small neighborhood $X$ of $x$ such that $\mathcal{I}_y=\J_y$ for all $y\in X\setminus \mathcal{W}$,
where $\mathcal{W}$ is a set of higher
codimension than $\Z$, the zero locus of $\J_x$. Then $\mathcal{I}_x=\J_x$. 
\end{lemma}

\begin{proof}
Let $Q_1\cap\ldots\cap Q_d$ be an irredundant primary decomposition of $\J_x$ and let $\Z_j$ be the zero locus of $Q_j$.
Since $\J_x$ has pure dimension there exist holomorphic sections $\psi_j$ such that
$\psi_j$ is not identically zero on $\Z_j$ but $\psi_j=0$ on $\mathcal{W}_j=\mathcal{W}\cap \Z_j$.
Assume for a contradiction that $\varphi\in\mathcal{I}_x\setminus\J_x$.
Then $\varphi\in\mathcal{I}_x\setminus Q_k$ for some $k$.
The $\holo_x$-module $\mathcal{I}_x/Q_k$ has support on $\mathcal{W}_k$ and since $\psi_k=0$ on $\mathcal{W}_k$ it follows from the Nullstellensatz that $\psi_k^M(\mathcal{I}_x/Q_k)=0$ for some positive integer $M$.
That is, $\psi_k^M\varphi\in Q_k$. Since $Q_k$ is primary and $\varphi\notin Q_k$ it follows that $\psi_k^M\in Q_k$ which
contradicts that $\psi$ is not identically zero on $\Z_j$.

\end{proof}

\begin{proof}[Proof of Theorem \ref{Thm}]
$\eqref{ett}\Rightarrow\eqref{tva}$: Assume that $\varphi\in\J_x$. Then, if $X$ is chosen small enough, there exists a
holomorphic section $\phi$ of $E_1$ such that $\varphi=f_1\phi$. In particular,
$\nabla_E\phi=\varphi$ and $\phi_p=0$, so with $\xi_{\ell}$ such that $f^*_{p+1}\xi_{\ell}=0$ we get
\[ 
\varphi\omega_{\ell}=[\xi_{\ell} u_p \varphi]_{\dbar}=[\xi_{\ell}
\phi_p]_{\dbar}=0,
\]
 where we have used \eqref{Tek} in the first equality and Lemma \ref{BevisLemma}
in the second one.
\\
$\eqref{ett}\Rightarrow\eqref{tva}$: Trivial.
\\
$\eqref{tre}\Rightarrow\eqref{ett}$:
Assume that $\varphi\Res_{\J}^{\ell}=0$ at $x$ for all $\ell$. Let $D$ be the neighborhood of $x$ so that $\varphi\Res_{\J}^{\ell}=0$ in $\mathcal{D}_\Z$.
Let $\mathcal{W}$ be the set of points $y$ in $\Z\cap D$ such that $\J_y\subset\holo_y$ is not Cohen-Macaulay. Then
$\mathcal{W}$ is a set of higher codimension than $\Z$ since $\J_y$ is Cohen-Macaulay for every smooth point $y\in\Z$.
Pick a point $y\in(\Z\setminus \mathcal{W})\cap D$ and let $\mathcal{U}\subset X$ be a
small neighborhood of $y$ in $D$ that does not intersect $\mathcal{W}$. If we restrict
the resolution $\holo(E_{\bullet})$ to $\mathcal{U}$, the minimal resolution
$\holo(E_{\bullet}^{\prime})$ of $\holo(E_0)/\J$ have length
$p=\operatorname{codim}\J_y=\operatorname{codim}\J_x$. 
By Proposition \ref{well-defined} we may change both the resolution
$\holo(E_{\bullet})$ and metric on $E_{\bullet}$ and still get
assumption \eqref{tre} for the form $u_p$ associated to those choices. Therefore we
may assume that $\holo(E_{\bullet})$ is on the form
$\holo(E_{\bullet}^{\prime}\oplus E_{\bullet}^{\prime\prime})$ where
$E_{\bullet}^{\prime}$ is minimal and $E_{\bullet}^{\prime\prime}$ is trivial.
We may also assume that the metric on $E_{\bullet}$
respects the direct product, and since $\holo(E_0)$ has rank $1$ we can write
$u_p\varphi=(u_p^{\prime}\oplus 0)\varphi$, where $u_p^{\prime}$ is associated to the minimal
resolution.

Even though the cohomology classes and hence the residues $\Res_{\J}^{\ell}$ change if we change generators $[\xi_{\ell}]$ of $\mathcal{H}^p(E^*_{\bullet})$, it follows from the definition of the cohomology classes $\omega_{\ell}$ that condition \eqref{tre} in Theorem~\ref{Thm} does not change.
Thus, in light of Remark ~\ref{gammal} we can change 
generators $[\xi_{\ell}]$ for
$\mathcal{H}^p(E_{\bullet}^*)$ and use Theorem~2.4 
in \cite{L} to see that
$\varphi\in\J_y$. If we do the procedure above for all points $y$ in
$\Z\setminus \mathcal{W}$ we get that $\varphi\in\J_x$ everywhere except at a set $\mathcal{W}$
of higher codimension. Lemma~\ref{lemma} now concludes the proof with $\mathcal{I}_x= \{
\varphi; \varphi\in\J_x \text{ on } X\setminus \mathcal{W} \}$.
\end{proof}

The bilinear map \eqref{map} induces the bilinear map 
\begin{equation}\label{nybilin}
\mathcal{H}^p(E_{\bullet}^*)\times\holo(E_0)/\J\to\F.
\end{equation}
Theorem~\ref{Thm} implies that \eqref{nybilin} is non-degenerate in the second factor.
The fact that \eqref{nybilin} is non-degenerate in the first factor follows from Theorem~1.2 in \cite{Andersson} that is based on residue currents (and hence Hironaka's theorem).
It may be possible to prove that it is non-degenerate in the first factor in
a similar, more elementary fashion, as in this paper but in order to do that one 
needs to prove a 
similar theorem as
Theorem~\ref{Thm} but for $E_0$ with higher rank.

The map
\begin{equation*}
\holo(E_0)/\J\to\mathcal{H}om(\mathcal{H}^p(E_{\bullet}^*),\F),
\end{equation*}
where $\varphi$ is mapped on the homomorphism defined by $[\xi]\mapsto[\xi
u_p\varphi]_{\dbar}$ for $[\xi]\in\mathcal{H}^p(E_{\bullet}^*)$  is closely
related to a theorem of Roos, \cite{Roos}, see \cite{Andersson} for details.

\section{Some examples}

\begin{example}
{\rm
If there is a resolution of $\holo/\J$ on the form 
\[0\rightarrow\holo^m\overset{g}{\rightarrow}\holo^r\overset{f}{\rightarrow}\holo\]
with the trivial metric, then
we get
\[\sigma_1=f^*(ff^*)^{-1} \quad\text{and}\quad \sigma_2=(g^*g)^{-1}g^*.\]
}
\end{example}
When the resolution has lenght greater or equal to $2$, the computations of the minimal inverses are much more involved, see \cite{W}. In that case one can sometimes use that it is known what $u$ is in the case of a complete intersection, cf., Section~\ref{intro}, and Theorem~1.1 in \cite{LW} to calculate $\omega_{\ell}$.

\begin{example}
{\rm
In two variables, let $\J_0=(x^2, xy, y^2)$. Then $\J_0$ is Cohen-Macaulay and we can chose $\xi_1$ and $\xi_2$ as the vectors $(1,0)$ and $(0,1)$. Using Theorem 1.1 in \cite{LW} we get
\[\omega_1 = \left[\frac{\bar{x}^2d\bar{y}-\bar{y}d\bar{x}^2}{(|x|^4+|y|^2)^2}\right]_{\dbar} \quad\text{and}\quad
\omega_2 = \left[\frac{\bar{y}^2d\bar{x}-\bar{x}d\bar{y}^2}{(|y|^4+|x|^2)^2}\right]_{\dbar}.\]
}
\end{example}
\begin{example}
{\rm
In three variables, let $\J_0$ be the Cohen-Macaulay ideal $(x^2,y^2,z^2,yz)$. Then
\[\omega_1 = \left[\frac{\bar{x}^2d\bar{y}\wedge d\bar{z}^2-\bar{y}d\bar{x}^2\wedge d\bar{z}^2 + \bar{z}^2d\bar{x}^2\wedge d\bar{y}}{(|x|^4+|y|^2 + |z|^4)^3}\right]_{\dbar}
\]
and
\[
\omega_2 = \left[\frac{\bar{x}^2d\bar{y}^2\wedge d\bar{z}-\bar{y}^2d\bar{x}^2\wedge d\bar{z} + \bar{z}d\bar{x}^2\wedge d\bar{y}^2}{(|x|^4+|y|^4 + |z|^2)^3}\right]_{\dbar}.
\]
}
\end{example}

\end{document}